\documentclass[a4paper]{amsart}
\usepackage{graphicx}
\usepackage{amssymb}
\usepackage{amsmath}
\usepackage{amsthm}
\usepackage{amscd}
\usepackage[all,2cell]{xy}

\UseAllTwocells \SilentMatrices
\newtheorem{thm}{Theorem}[section]

\newtheorem{cor}[thm]{Corollary}
\newtheorem{lem}[thm]{Lemma}
\newtheorem{exm}[thm]{Example}

\newtheorem{prop}[thm]{Proposition}
\theoremstyle{definition}

\theoremstyle{remark}
\newtheorem{rem}[thm]{\bf Remark}
\numberwithin{equation}{section}

\begin{document}
\title[Singular equivalences induced by homological epimorphisms]
{Singular equivalences induced by homological epimorphisms}

\author[  Xiao-Wu Chen
] {Xiao-Wu Chen}

\thanks{The author is supported by Special Foundation of President of The Chinese Academy of Sciences
(No.1731112304061) and  National Natural Science Foundation of China (No.10971206).}
\date{\today}

\thanks{E-mail:
xwchen$\symbol{64}$mail.ustc.edu.cn}

\maketitle

\dedicatory{}%
\commby{}%

\begin{abstract}
We prove that a certain homological epimorphism between two algebras induces  a triangle equivalence
between their singularity categories. Applying the result to a construction of matrix algebras, we describe
the singularity categories of some non-Gorenstein algebras.
\end{abstract}

\section{Introduction}

Let $A$ be a finite dimensional algebra over a field $k$. Denote by $A\mbox{-mod}$ the
category of finitely generated left $A$-modules, and by $\mathbf{D}^b(A\mbox{-mod})$ the bounded derived category.
Following \cite{Or04}, the \emph{singularity category} $\mathbf{D}_{\rm sg}(A)$ of $A$ is the Verdier quotient
triangulated category of $\mathbf{D}^b(A\mbox{-mod})$ with respect to the full subcategory
formed by perfect complexes; see also \cite{Buc, KV, Hap91, Ric, Bel2000, Kra} and \cite{Ch3}.

The singularity category  measures the homological singularity of an algebra:  the algebra
$A$ has finite global dimension if and only if  its singularity category $\mathbf{D}_{\rm sg}(A)$
is trivial. Meanwhile, the singularity category captures the stable homological features of an algebra (\cite{Buc}).

A fundamental result of Buchweitz and Happel states that for a Gorenstein algebra $A$, the singularity category $\mathbf{D}_{\rm sg}(A)$ is triangle equivalent to the stable category of (maximal) Cohen-Macaulay $A$-modules (\cite{Buc} and \cite{Hap91}). This result specializes to Rickard's result (\cite{Ric}) on self-injective algebras.  For non-Gorenstein algebras,
not much is known about their singularity categories (\cite{Ch09, Ch11}).

The following concepts might be useful in the study of singularity categories. Two algebras
$A$ and $B$ are said to be \emph{singularly equivalent} provided that there is a triangle equivalence
between $\mathbf{D}_{\rm sg}(A)$ and $\mathbf{D}_{\rm sg}(B)$. Such an equivalence is  called a
\emph{singular equivalence}; compare \cite{Or06}. In this case, if $A$ is non-Gorenstein and $B$ is Gorenstein, then
Buchweitz-Happel's theorem applies to give a description of $\mathbf{D}_{\rm sg}(A)$ in terms of Cohen-Macaulay modules
over $B$. We observe that a derived equivalence of two algebras, that is, a triangle equivalence between their bounded derived categories, induces naturally
a singular equivalence. The converse is not true in general.

Let $A$ be an algebra and let $J\subseteq A$ be a two-sided ideal. Following \cite{PeXi}, we call $J$ a \emph{homological ideal}
provided that the canonical map $A\rightarrow A/J$ is a homological epimorphism (\cite{GL}), meaning that the naturally induced
functor ${\bf D}^b(A/J\mbox{-mod})\rightarrow {\bf D}^b(A\mbox{-mod})$ is fully faithful.

The main observation we make is as follows.
\vskip 5pt

\noindent {\bf Theorem}. \emph{Let $A$ be a finite dimensional $k$-algebra and let $J\subseteq A$ be a
homological ideal which has finite projective dimension as an $A$-$A$-bimodule. Then
there is a singular equivalence between $A$ and $A/J$.}

\vskip 5pt

The paper is structured as follows. In Section 2, we recall some ingredients and then prove Theorem. In Section 3, we
apply Theorem to a construction of  matrix algebras, and then describe the
singularity categories of some non-Gorenstein algebras. In particular, we give two examples,
which extend in different manners an example considered by Happel in \cite{Hap91}.

\section{Proof of Theorem}

We will present the proof of Theorem in this section. Before that, we recall from \cite{Ver77} and \cite{Har66} some
known results on triangulated categories and derived categories.

Let $\mathcal{T}$ be a triangulated category. We will denote its translation functor by $[1]$.
For a triangulated subcategory $\mathcal{N}$, we denote by $\mathcal{T}/\mathcal{N}$ the Verdier
quotient triangulated category. The quotient functor $q\colon \mathcal{T}\rightarrow \mathcal{T}/\mathcal{N}$
has the property that $q(X)\simeq 0$ if and only if $X$ is a direct summand of an object in $\mathcal{N}$.
In particular, if $\mathcal{N}$ is a \emph{thick} subcategory, that is, it is closed under direct summands,
we have that ${\rm Ker}\; q=\mathcal{N}$. Here, for a triangle functor $F$, ${\rm Ker}\; F$ denotes
the (thick) triangulated subcategory consisting of objects on which $F$ vanishes.

The following result is well known.

\begin{lem}\label{lem:1}
Let $F\colon \mathcal{T}\rightarrow \mathcal{T}'$ be a triangle functor which allows a fully faithful right adjoint
$G$. Then $F$ induces uniquely a triangle equivalence $\mathcal{T}/{{\rm Ker}\; F}\simeq \mathcal{T}'$.
\end{lem}

\begin{proof}
The existence of the induced functor follows from the universal property of the quotient functor.
The result is a triangulated version of \cite[Proposition I. 1.3]{GZ}. For details, see \cite[Propositions 1.5 and 1.6]{BK}.
\end{proof}

Let $F\colon \mathcal{T}\rightarrow \mathcal{T}'$ be a triangle functor. Assume that $\mathcal{N}\subseteq \mathcal{T}$
and $\mathcal{N}'\subseteq \mathcal{T}'$ are triangulated subcategories satisfying $F\mathcal{N}\subseteq \mathcal{N}'$. Then there is a uniquely induced
triangle functor $\bar{F}\colon \mathcal{T}/\mathcal{N}\rightarrow \mathcal{T}'/{\mathcal{N}'}$.

\begin{lem}\label{lem:2}{\rm (\cite[Lemma 1.2]{Or04})}
Let $F\colon \mathcal{T}\rightarrow \mathcal{T}'$ be a triangle functor which has a right adjoint $G$.
Assume that $\mathcal{N}\subseteq \mathcal{T}$ and $\mathcal{N}'\subseteq \mathcal{T}'$ are triangulated
subcategories satisfying that $F\mathcal{N}\subseteq \mathcal{N}'$ and $G\mathcal{N}'\subseteq \mathcal{N}$.
Then the induced functor $\bar{F}\colon \mathcal{T}/\mathcal{N}\rightarrow \mathcal{T}'/{\mathcal{N}'}$ has a
right adjoint $\bar{G}$. Moreover, if $G$ is fully faithful, so is $\bar{G}$.
\end{lem}

\begin{proof}
The unit and counit of $(F, G)$ induce uniquely two natural transformations ${\rm Id}_{\mathcal{T}/{\mathcal{N}}}\rightarrow
\bar{G}\bar{F}$ and $\bar{F}\bar{G}\rightarrow {\rm Id}_{\mathcal{T}'/\mathcal{N}'}$, which are the corresponding unit and
counit of the adjoint pair $(\bar{F}, \bar{G})$; consult \cite[Chapter IV, Section 1, Theorem 2(v)]{MacL}. Note that
the fully-faithfulness of $G$ is equivalent to that the counit of $(F, G)$ is an isomorphism. It follows that
the counit of $(\bar{F}, \bar{G})$ is also an isomorphism, which is equivalent to the fully-faithfulness of $\bar{G}$; consult \cite[Chapter IV, Section 3, Theorem 1]{MacL}.
\end{proof}

 Let $k$ be a field and let $A$ be a finite dimensional $k$-algebra. Recall that $A\mbox{-mod}$
 is the category of finite dimensional left $A$-modules.  We write
$_AA$ for the regular left $A$-module. Denote by $\mathbf{D}(A\mbox{-mod})$  (\emph{resp}. $\mathbf{D}^b(A\mbox{-mod})$)
the (\emph{resp}. bounded) derived category of $A\mbox{-mod}$.  We identify $A\mbox{-mod}$ as the full subcategory of $\mathbf{D}^b(A\mbox{-mod})$ consisting of stalk complex concentrated at degree zero; see \cite[Proposition I. 4.3]{Har66}.

 A complex of $A$-modules is usually denoted by $X^{\bullet}=(X^n, d^n)_{n\in \mathbb{Z}}$, where $X^n$ are
 $A$-modules and the differentials $d^n\colon X^n\rightarrow X^{n+1}$ are homomorphisms of modules satisfying $d^{n+1}\circ d^n=0$.
 Recall that a complex in $\mathbf{D}^b(A\mbox{-mod})$ is \emph{perfect} provided that it is isomorphic to a bounded complex consisting of projective modules. The full subcategory
consisting of perfect complexes is denoted by ${\rm perf}(A)$. Recall from \cite[Lemma 1.2.1]{Buc}  that
a complex $X^\bullet$ in $\mathbf{D}^b(A\mbox{-{\rm mod}})$ is perfect if and only if there is a natural number
$n_0$ such that for each $A$-module $M$, ${\rm Hom}_{\mathbf{D}^b(A\mbox{-}{\rm mod})}(X^\bullet, M[n])=0$ for all $n\geq n_0$.
  It follows that ${\rm perf}(A)$ is  a thick  subcategory of $\mathbf{D}^b(A\mbox{-mod})$.  Indeed, it is the smallest thick subcategory of $\mathbf{D}^b(A\mbox{-mod})$ containing $_AA$.

Let $\pi\colon A\rightarrow B$ be a homomorphism of algebras. The functor of restricting of scalars
$\pi^*\colon B\mbox{-mod}\rightarrow A\mbox{-mod}$ is exact, and it extends to a triangle functor
$\mathbf{D}^b(B\mbox{-mod})\rightarrow \mathbf{D}^b(A\mbox{-mod})$, which will still be denoted by
$\pi^*$. Following \cite{GL}, we call  the homomorphism $\pi$ a \emph{homological epimorphism} provided that
$\pi^*\colon \mathbf{D}^b(B\mbox{-mod})\rightarrow \mathbf{D}^b(A\mbox{-mod})$ is fully faithful. By \cite[Theorem 4.1(1)]{GL} this is equivalent to that $\pi\otimes^\mathbf{L}_A B\colon B \simeq A\otimes_A^\mathbf{L} B\rightarrow B\otimes^\mathbf{L}_AB$ is an isomorphism in $\mathbf{D}(A^e\mbox{-mod})$. Here, $A^e=A\otimes_k A^{\rm op}$ is the enveloping algebra of $A$, and
we identify $A^e\mbox{-mod}$ as the category of $A$-$A$-bimodules.

\begin{lem}\label{lem:3}{\rm (\cite[Proposition 2.2(a)]{PeXi})}
Let $J\subseteq A$ be an ideal and let $\pi\colon A\rightarrow A/J$ be the canonical projection. Then
$\pi$ is a homological epimorphism if and only if $J^2=J$ and ${\rm Tor}_i^A(J, A/J)=0$ for all $i\geq 1$.
\end{lem}

In the situation of the lemma, the ideal $J$ is called a \emph{homological ideal} in \cite{PeXi}. As a special case,
we call an ideal  $J$  a \emph{hereditary ideal} provided that $J^2=J$ and $J$ is a projective $A$-$A$-bimodule; compare \cite[Lemma 3.4]{PeXi}.

\begin{proof}
The natural  exact sequence $0\rightarrow J\rightarrow A\rightarrow A/J\rightarrow 0$ of $A$-$A$-bimodules induces
a triangle $J\rightarrow A\rightarrow A/J\rightarrow J[1]$ in $\mathbf{D}^b(A^e\mbox{-mod})$. Applying the
functor $-\otimes^\mathbf{L}_A A/J$, we get  a triangle $J\otimes^\mathbf{L}_A A/J \rightarrow A/J \rightarrow A/J\otimes^\mathbf{L}_A A/J\rightarrow J\otimes^\mathbf{L}_A A/J[1]$. Then $\pi$ is a homological epimorphism, or equivalently $\pi\otimes_A A/J$ is an isomorphism if and only if
$J\otimes^\mathbf{L}_A A/J=0$; see \cite[Lemma I.1.7]{Hap88}. This is equivalent to that  ${\rm Tor}_i^A(J, A/J)=0$ for all $i\geq 0$. We note that ${\rm Tor}_0^A(J, A/J)\simeq J\otimes_AA/J\simeq J/J^2$.
\end{proof}

Now we are in the position to prove Theorem. Recall that for an algebra $A$, its singularity category $\mathbf{D}_{\rm sg}(A)=\mathbf{D}^b(A\mbox{-mod})/{{\rm perf}(A)}$. Moreover, a complex $X^\bullet$ becomes zero in $\mathbf{D}_{\rm sg}(A)$
if and only if it is perfect. Here, we use the fact that ${\rm perf}(A)\subseteq \mathbf{D}^b(A\mbox{-mod})$ is a thick
subcategory.

\vskip 10pt

\noindent {\bf Proof of Theorem.}\quad Write $B=A/J$.  Since $J$, as an $A$-$A$-bimodule, has finite projective dimension,
so it  has finite projective dimension both as a left and right $A$-module. Consider the natural exact sequence
$0\rightarrow J\rightarrow A\rightarrow B\rightarrow 0$. It follows that $B$, both as a left and right $A$-module, has
finite projective dimension. Moreover, for a complex $X^\bullet$ in $\mathbf{D}^b(A\mbox{-mod})$, $J\otimes^\mathbf{L}_AX^\bullet$
is perfect. Indeed, take a bounded projective resolution $P^\bullet\rightarrow J$ as an $A^e$-module. Then $J\otimes^\mathbf{L}_AX^\bullet\simeq P^\bullet \otimes_A X^\bullet$. This is a perfect complex, since each left $A$-module $P^i\otimes_AX^j$ is projective.

Denote by $\pi\colon A\rightarrow B$ be the canonical projection. By the assumption,  the functor $\pi^*\colon \mathbf{D}^b(B\mbox{-mod})\rightarrow \mathbf{D}^b(A\mbox{-mod})$ is fully faithful.
Since  $\pi^*(B)$ is perfect, the functor $\pi^*$ sends perfect complexes to perfect complexes.
Then it induces a triangle functor $\bar{\pi^*}\colon \mathbf{D}_{\rm sg}(B)\rightarrow \mathbf{D}_{\rm sg}(A)$. We will show
that $\bar{\pi^*}$ is an equivalence.

 The functor $\pi^*\colon \mathbf{D}^b(B\mbox{-mod})\rightarrow \mathbf{D}^b(A\mbox{-mod})$ has a left adjoint $F=B\otimes_A^\mathbf{L}-\colon \mathbf{D}^b(A\mbox{-mod})\rightarrow \mathbf{D}^b(B\mbox{-mod})$. Here we use
 the fact that the right $A$-module $B$ has finite projective dimension.  Since $F$ sends perfect complexes
 to perfect complexes, we have the induced triangle functor $\bar{F}\colon  \mathbf{D}_{\rm sg}(A)\rightarrow \mathbf{D}_{\rm sg}(B)$. By Lemma \ref{lem:2} we have the adjoint pair $(\bar{F}, \bar{\pi^*})$; moreover, the functor $\bar{\pi^*}$ is fully
 faithful. By Lemma \ref{lem:1} there is a triangle equivalence $\mathbf{D}_{\rm sg}(A)/{{\rm Ker}\; \bar{F}}\simeq \mathbf{D}_{\rm sg}(B)$.

 It remains to show that the essential kernel ${\rm Ker}\; \bar{F}$ is trivial. For this, assume
 that a complex $X^\bullet$ lies in ${\rm Ker}\; \bar{F}$. This means that the complex $F(X^\bullet)$ in $\mathbf{D}^b(B\mbox{-mod})$ is perfect. Since
 $\pi^* $ preserves perfect complexes, it follows that $\pi^*F(X^\bullet)$ is also perfect. The natural exact sequence
$0\rightarrow J\rightarrow A\rightarrow B\rightarrow 0$ induces a triangle $J\otimes^\mathbf{L}_A X^\bullet \rightarrow
X^\bullet \rightarrow \pi^*F(X^\bullet) \rightarrow J\otimes^\mathbf{L}_A X^\bullet [1]$ in $\mathbf{D}^b(A\mbox{-mod})$.
Recall that $J\otimes^\mathbf{L}_AX^\bullet$ is perfect. It follows that $X^\bullet$ is perfect, since ${\rm perf}(A)\subseteq
\mathbf{D}^b(A\mbox{-mod})$ is a triangulated subcategory. The proves that $X^\bullet$ is zero in $\mathbf{D}_{\rm sg}(A)$. \hfill
$\square$

\vskip 10pt

The following special case of Theorem is of interest.

\begin{cor}\label{cor:1}
Let $A$ be a finite dimensional algebra and $J\subseteq A$ a hereditary ideal. Then we have a triangle equivalence $\mathbf{D}_{\rm sg}(A)\simeq \mathbf{D}_{\rm sg}(A/J)$.
\end{cor}

\begin{proof}
It suffices to observe by Lemma \ref{lem:3} that $J$ is a homological ideal.
\end{proof}

\section{Examples}

In this section, we will describe a construction of matrix algebras  to illustrate Corollary \ref{cor:1}. In particular, the
singularity categories of sone non-Gorenstein algebras are studied.

The following construction is similar to \cite[Section 4]{KN}. Let $A$ be a finite dimensional algebra over a field $k$. Let $_AM$ and $N_A$ be a left  and right $A$-module, respectively. Then
$M\otimes_kN$ becomes an $A$-$A$-bimodule. Consider an $A$-$A$-bimodule monomorphism $\phi\colon M\otimes_k N\rightarrow A$
such that ${\rm Im}\; \phi$ vanishes both on $M$ and $N$. Here, we observe that ${\rm Im}\; \phi\subseteq A$ is an ideal.
The matrix $\Gamma=\begin{pmatrix}A & M \\ N& k\end{pmatrix}$
becomes an associative algebra via the following multiplication
$$\begin{pmatrix}a & m \\ n & \lambda\end{pmatrix} \;  \begin{pmatrix}a' & m' \\ n' & \lambda'\end{pmatrix}= \begin{pmatrix}aa'+\phi(m\otimes n) & am'+\lambda' m \\ na'+\lambda n' & \lambda\lambda'\end{pmatrix}.$$

\begin{prop}\label{prop:1}
Keep the notation and assumption as above. Then there is a triangle equivalence $\mathbf{D}_{\rm sg}(\Gamma) \simeq \mathbf{D}_{\rm sg}(A/{\rm Im}\; \phi)$.
\end{prop}

\begin{proof}
Set $J=\Gamma e\Gamma$ with $e=\begin{pmatrix}0 & 0 \\ 0 & 1\end{pmatrix}$. Observe that $\Gamma/{J=A/{\rm Im}\; \phi}$. The ideal $J$ is hereditary: $J^2=J$ is clear, while the natural map  $\Gamma e\otimes_k e\Gamma\rightarrow J$ is an isomorphism of  $\Gamma$-$\Gamma$-bimodules and then $J$ is a projective $\Gamma$-$\Gamma$-bimodule. The isomorphism uses that $\phi$ is mono. Then
we apply Corollary \ref{cor:1}.
\end{proof}

\begin{rem}
The above construction contains the one-point extension and coextension of algebras,
where $M$ or $N$ is zero. Hence Proposition \ref{prop:1} contains the results in \cite[Section 4]{Ch11}.
\end{rem}

We will illustrate Proposition \ref{prop:1} by three examples. Two of these examples extend
an example considered by Happel in \cite{Hap91}. In particular, based on results in 
\cite{Ch11}, we obtain descriptions of the singularity categories of some non-Gorenstein algebras.

Recall from \cite{Hap91} that an algebra $A$ is
\emph{Gorenstein} provided that both as a left and right module, the regular module $A$ has finite injective dimension. It follows from \cite[Theorem 4.4]{Buc} and \cite[Theorem 4.6]{Hap91} that in the Gorenstein case, the singularity
category $\mathbf{D}_{\rm sg}(A)$ is \emph{Hom-finite}. This means that all  Hom spaces  in $\mathbf{D}_{\rm sg}(A)$ are  finite dimensional over $k$.

For algebras given by quivers and relations, we refer to \cite[Chapter III]{ARS}.

\begin{exm}
{\rm
Let $\Gamma$ be the $k$-algebra given by  the following quiver $Q$ with relations $\{\delta x, \beta x, x\gamma, x\alpha, \beta \gamma, \delta \alpha, \beta \alpha, \delta \gamma, \alpha \beta-\gamma\delta\}$.  We write the concatenation of paths from the left to the right.
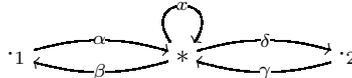
\SelectTips{eu}{10}
\[\xymatrix@!=7pt{
\cdot_1 \ar@/^/[rr]|{\alpha} & &\ar@/^/[ll]|{\beta} \ast \ar@/^/[rr]|{\delta} \ar@(lu,ru)[]|{x}  && \cdot_2 \ar@/^/[ll]|{\gamma}}\]

We have in $\Gamma$ that $1=e_1+e_\ast+ e_2$, where the $e$'s are the primitive idempotents
corresponding to the vertices. Set $\Gamma'=\Gamma/\Gamma e_1\Gamma$. It is an algebra with radical square zero, whose quiver is obtained from $Q$ by removing the vertex $1$ and the adjacent arrows.

We identify $\Gamma$ with  $\begin{pmatrix}A & k\alpha \\ k\beta & k\end{pmatrix}$, where the $k$ in the southeast corner
is identified with $e_1\Gamma e_1$, and $A=(1-e_1)\Gamma (1-e_1)$. The corresponding ${\rm Im}\; \phi$ equals $k\alpha\beta$, and
we have $A/{{\rm Im}\; \phi}=\Gamma'$; consult the proof of Proposition \ref{prop:1}. Then Proposition \ref{prop:1} yields  a triangle equivalence $\mathbf{D}_{\rm sg}(\Gamma)\simeq \mathbf{D}_{\rm sg}(\Gamma')$.

The triangulated category $\mathbf{D}_{\rm sg}(\Gamma')$ is completely described in \cite{Ch11}; also see \cite{Sm}; in particular, it is
not Hom-finite.  More precisely, it is equivalent to the category of finitely generated projective modules on a von Neumann regular algebra. The algebra $\Gamma'$, or rather its Koszul dual,  is related to the noncommutative space of Penrose tiling via the work of Smith; see \cite[Theorem 7.2 and Example]{Sm}. We point out that the algebra $\Gamma$ is non-Gorenstein, since $\mathbf{D}_{\rm sg}(\Gamma)$ is not Hom-finite.
}\end{exm}

\begin{exm}\label{exm:2}
{\rm

Let $\Gamma$ be the $k$-algebra given by the following quiver $Q$ with relations $\{
x_1^2, x_2^2, x_1\alpha_1, x_2\alpha_1, \beta_2\alpha_1, \beta_2\alpha_1, x_1\alpha_2,
x_2\alpha_2, \beta_1\alpha_2, \beta_2\alpha_2, \alpha_1\beta_1-x_1x_2,
\alpha_2\beta_2-x_2x_1\}$.

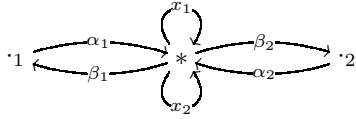
\SelectTips{eu}{10}
\[\xymatrix@!=7pt{
\cdot_1 \ar@/^/[rr]|{\alpha_1} & &\ar@/^/[ll]|{\beta_1} \ast \ar@/^/[rr]|{\beta_2} \ar@(lu,ru)[]|{x_1} \ar@(ld,rd)[]|{x_2} && \cdot_2 \ar@/^/[ll]|{\alpha_2}}\]

We claim that there is a triangle equivalence $\mathbf{D}_{\rm sg}(\Gamma)\simeq \mathbf{D}_{\rm sg}(k\langle x_1, x_2\rangle/{(x_1, x_2)^2})$. Here, $k\langle x_1, x_2\rangle$ is the free algebra with two variables.

We point out that the triangulated category $\mathbf{D}_{\rm sg}(k\langle x_1, x_2\rangle/{(x_1, x_2)^2})$ is described
completely in \cite[Example 3.11]{Ch11}. Similar as in the example above, this algebra $\Gamma$ is non-Gorenstein.

To see the claim,  we observe that the quiver $Q$ has two loops and two $2$-cycles. The proof is done by ``removing the $2$-cycles".
We have  a natural isomorphism $\Gamma=\begin{pmatrix} A & k\alpha_1 \\ k\beta_1 & k\end{pmatrix}$, where $k=e_1\Gamma e_1$ and $A=(1-e_1)\Gamma(1-e_1)$. We observe that
Proposition \ref{prop:1} applies with the corresponding ${\rm Im}\; \phi=k\alpha_1\beta_1$.
Set $A/{{\rm Im}\; \phi}=\Gamma'$. So $\mathbf{D}_{\rm sg}(\Gamma)\simeq \mathbf{D}_{\rm sg}(\Gamma')$. The quiver of $\Gamma'$ is obtained from $Q$ by removing the vertex $1$  and the adjacent arrows, while its relations are obtained from the ones of $\Gamma$ by replacing $\alpha_1\beta_1-x_1x_2$ with $x_1x_2$.
Similarly, $\Gamma'=\begin{pmatrix}A' & k\alpha_2\\ k\beta_2 &k\end{pmatrix}$ with $k=e_2\Gamma' e_2$ and $A'=e_*\Gamma' e_*$.
Then Proposition \ref{prop:1} applies and we get the equivalence $\mathbf{D}_{\rm sg}(\Gamma')\simeq \mathbf{D}_{\rm sg}(k\langle x_1, x_2\rangle/{(x_1, x_2)^2})$.

This example generalizes directly to a quiver with $n$ loops and $n$ $2$-cycles with similar relations. The corresponding
statement for the case $n=1$ is implicitly contained in \cite[2.3 and 4.8]{Hap91}.
}\end{exm}

The last example is a Gorenstein algebra.

\begin{exm}
{\rm Let $r\geq 2$. Consider the following quiver $Q$ consisting of three $2$-cycles and a central $3$-cycle $Z_3$.
We identify $\gamma_3$ with $\gamma_0$, and denote by $p_i$ the path in the central cycle starting at vertex $i$ of
length $3$.
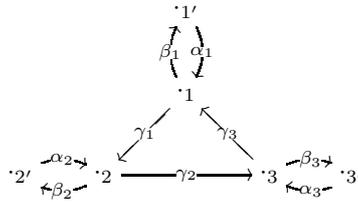
\SelectTips{eu}{10}
\[\xymatrix@!=7pt{
&  &  \cdot_{1'} \ar@/^/[d]|{\alpha_1} \\
& &  \cdot_{1} \ar@/^/[u]|{\beta_1}\ar[ld]|{\gamma_1} \\
\cdot_{2'} \ar@/^/[r]|{\alpha_2} & \ar@/^/[l]|{\beta_2} \cdot_2  \ar[rr]|{\gamma_2} & & \cdot_3 \ar[lu]|{\gamma_3} \ar@/^/[r]|{\beta_3} & \cdot_{3'} \ar@/^/[l]|{\alpha_3}
}\]
Let $\Gamma$ be the $k$-algebra given by the quiver $Q$ with relations $\{  \beta_i\alpha_i, \gamma_i\alpha_i, \beta_i\gamma_{i-1},
\beta_i\alpha_i-p_i^r\; |\; i=1, 2,3\}$. We point out that in $\Gamma$ all paths in the central cycle of length
strictly larger than $3r+1$ vanish.

Set $A=kZ_3/{(\gamma_1, \gamma_2, \gamma_3)^{3r}}$, where $kZ_3$ is the path algebra  of the central $3$-cycle $Z_3$.
The algebra $A$ is self-injective and  Nakayama (\cite[p.111]{ARS}). Denote by $A\mbox{-\underline{mod}}$ the stable category of $A$-modules; it is naturally a triangulated category; see \cite[Theorem I.2.6]{Hap88}.

We claim that there is a triangle equivalence $\mathbf{D}_{\rm sg}(\Gamma)\simeq A\mbox{-\underline{mod}}$.

For the claim, we observe an isomorphism  $A=\Gamma/{\Gamma(e_{1'}+e_{2'}+e_{3'})\Gamma}$. We argue as in Example \ref{exm:2} by removing the three $2$-cycles and applying Proposition \ref{prop:1} repeatedly. Then we get a triangle equivalence $\mathbf{D}_{\rm sg}(\Gamma)\simeq \mathbf{D}_{\rm sg}(A)$. Finally, by \cite[Theorem 2.1]{Ric} we have a triangle equivalence $\mathbf{D}_{\rm sg}(A)\simeq A\mbox{-\underline{mod}}$. Then we are done.

We observe that the algebra $\Gamma$ is Gorenstein with self injective dimension two. Moreover, this example generalizes directly to a quiver with $n$ $2$-cycles and a central $n$-cycle with similar
relations. The case where $n=1$ and $r=2$ coincides with the example considered in \cite[2.3 and 4.8]{Hap91}.
}\end{exm}

\vskip 10pt

 \noindent {\bf Acknowledgements.}\quad The results of this paper answer partially a question, which was raised by
 Professor Changchang Xi during a conference held in Jinan, June 2011. The author thanks Huanhuan Li for helpful discussion.

\bibliography{}

\vskip 10pt

 {\footnotesize \noindent Xiao-Wu Chen, School of Mathematical Sciences,
  University of Science and Technology of
China, Hefei 230026, Anhui, PR China \\
Wu Wen-Tsun Key Laboratory of Mathematics, USTC, Chinese Academy of Sciences, Hefei 230026, Anhui, PR China.\\
URL: http://mail.ustc.edu.cn/$^\sim$xwchen}

\end{document}